\documentclass[12pt]{amsart}

\usepackage{amsmath, amssymb, amsthm, fullpage}

\newcommand{\Int}{\operatorname{int}}
\newcommand{\bC}{\mathbb{C}}
\newcommand{\bD}{\mathbb{D}}
\newcommand{\bH}{\mathbb{H}}
\newcommand{\bN}{\mathbb{N}}

\newcommand{\bR}{\mathbb{R}}
\newcommand{\bZ}{\mathbb{Z}}

\theoremstyle{plain}
\newtheorem{lemma}{Lemma}
\newtheorem{claim}{Claim}
\newtheorem*{mainth}{Theorem}
\newtheorem*{iversenth}{Iversen's Theorem}

\theoremstyle{definition}
\newtheorem*{acknowledgement}{Acknowledgement}
\theoremstyle{remark}

\begin{document} 
 
\title[A generalization of a completeness lemma]{A generalization of a completeness lemma in minimal surface theory}
\author[Y\^usuke Okuyama]{Y\^usuke Okuyama}
\address{
Division of Mathematics,
Kyoto Institute of Technology,
Sakyo-ku, Kyoto 606-8585 Japan.}
\email{okuyama@kit.ac.jp}

\author[Katsutoshi Yamanoi]{Katsutoshi Yamanoi}
\address{
Department of Mathematics,
Tokyo Institute of Technology,
Oh-okayama, Meguro-ku,
Tokyo 152-8551
Japan.}
\email{yamanoi@math.titech.ac.jp}

\date{\today}



\begin{abstract}
 We settle a question posed by Umehara and Yamada, which generalizes
 a completeness lemma useful in differential geometry.
\end{abstract}

\maketitle

The following answers affirmatively a question posed by Umehara and Yamada
\cite[Question C]{UY}.

\begin{mainth}\label{th:spiral}
 Let $f$ be a holomorphic function on $\{|\zeta|>1\}\subset\bC$ such that 
 $f(\{|\zeta|>1\})\subset\bC\setminus\{0\}$ and let 
 $n$ be a non-negative integer.
If every real-analytic curve $\gamma:[0,1)\to\{|\zeta|>1\}$ tending to $\infty$ satisfies
 \begin{gather}
  \int_\gamma|\log\zeta|^n|f(\zeta)|| d\zeta|=\infty,\label{eq:complete}
 \end{gather}
 then $f$ is meromorphic at $\infty$. 
\end{mainth}
In the special case of $n=0$, this Theorem
reduces to the completeness lemma
due to MacLane and Voss (cf.\ Osserman \cite[p.\ 89]{O}), 
which plays an important role in minimal surface theory.
A new insight by Umehara and Yamada
is the possibility to take into account the variation of the argument of the curve $\gamma$, 
namely the imaginary part of $\log \gamma$, motivated by their investigation of 
parabolic ends of constant mean curvature one surfaces in de Sitter 3-space.
A notable consequence of Theorem is an affirmative answer to \cite[Question 2]{UYc}.
This implication is due to Umehara and Yamada \cite{UYc}.
For more details and backgrounds, we refer \cite{UY} and \cite{UYc}.
Our proof is based on the theory of entire functions, 
i.e., holomorphic functions on $\bC$, 
and of harmonic measures, while the problem has its origin in differential geometry. 

\par

After having written this article, we learned that the Theorem, except for the real-analyticity of the path $\gamma$, could be shown using Huber's result \cite{Huber57}.
Our proof of the Theorem is an improvement of the argument in Osserman's book \cite{O}.

\begin{proof}
 Since the integral in (\ref{eq:complete}) is non-decreasing as 
 $n\ge 0$ increases,
 the Theorem for $n=0$ is a consequence of that for $n>0$. We assume that 
 $n$ is a positive integer.

 We reduce the problem to the case where $f$ is defined over all of $\bC$
 and satisfies $f(\bC\setminus\{0\})\subset\bC\setminus\{0\}$
 (cf.\ Osserman \cite[p.\ 89]{O}).
 Let us consider possibly multivalued holomorphic functions $\log f(\zeta)$ and
 $\log\zeta$ on $\{|\zeta|>1\}$, and choose $k\in\bZ$ such that
 $\log f(\zeta)-k\log\zeta$ is single-valued and holomorphic on 
 $\{|\zeta|>1\}$.
 Then its Laurent expansion is written as
 \begin{gather*}
  \log f(\zeta)-k\log\zeta=\sum_{j=-\infty}^{\infty}c_j\zeta^j=
  H(\zeta)+h(\zeta),
 \end{gather*}
 where $H(\zeta)=\sum_{j=0}^\infty c_j\zeta^j$ is an entire function
 and $h(\zeta)=\sum_{j=-\infty}^{-1}c_j\zeta^j$ is a holomorphic function 
 on $\{|\zeta|>1\}\cup\{\infty\}\subset\hat{\bC}$ such that $h(\infty)=0$. Hence 
 \begin{gather*}
  f(\zeta)=e^{h(\zeta)}\zeta^ke^{H(\zeta)}.
 \end{gather*}

Set $g(\zeta)=\zeta ^ke^{H(\zeta)}$ if $k\geq 0$, and $g(\zeta )=e^{H(\zeta)}$ if $k<0$.
Then $g$ is entire and $f(\zeta )/g(\zeta )$ is holomorphic near $\infty$.
The condition (\ref{eq:complete}) implies
 \begin{gather*}
  \int_\gamma|\log\zeta|^n|g(\zeta )| | d\zeta|=\infty,
\end{gather*}
and if $g$ is meromorphic at $\infty$, then so is $f$.
Hence replacing $f$ with $g$ if necessary,
we may assume that $f$ is an entire 
function on $\bC$ and that $f(\bC\setminus\{0\})\subset\bC\setminus\{0\}$.

Consider an indefinite integral
\begin{gather*}
 G(\zeta )=\int_0^\zeta f(\zeta )d\zeta
\end{gather*}
of $f$. Since the zeros of $G$ are isolated, there is $a>0$ such that
\begin{gather*}
 \min_{t\in\bR}|G(e^{a+it})|>0.
\end{gather*}
We fix $a>0$ with this property throughout.
Set
\begin{gather*}
 F(z):=\int_a^{a+z}z^nf(e^z)e^z dz.
\end{gather*}
Then we note from our earlier discussion that $-a$ is the only critical point of $F$, that is,
\begin{gather*}
 \{z\in\bC;F'(z)=0\}=\{-a\}.
\end{gather*}
Let $\zeta=\pi(z):=e^{z+a}$ be a covering map from $\bC$ to $\bC\setminus\{0\}$.
For any real-analytic curve $\Gamma$ in $z$-plane, we have
\begin{gather}\label{eqn:length}
\int_{\pi\circ\Gamma}|\log\zeta|^n|f(\zeta)|| d\zeta|
=(\text{the Euclidean length of }F\circ\Gamma).
\end{gather}

\begin{lemma}\label{th:imaginary}
$\lim_{t\in\bR,|t|\to\infty}|F(it)|=\infty$.
\end{lemma}

\begin{proof}
 Let us define $n+2$ auxiliary entire functions $G_0,G_1,\cdots,G_{n+1}$ 
inductively; put 
\begin{gather*}
 G_0(\zeta):=f(\zeta)\cdot\zeta
\end{gather*}
 and for each $j\in\{0,1,\ldots,n\}$,
\begin{gather*}
 G_{j+1}(\zeta):=\int_0^\zeta \frac{G_j(\zeta)}{\zeta} d\zeta.
\end{gather*}
Then $G_1(\zeta )=G(\zeta )$, and for every $j\in\{0,1,\ldots,n\}$,
\begin{gather*}
 \frac{ d G_{j+1}(e^{a+it})}{ d t}=iG_j(e^{a+it}).
\end{gather*}
Hence for $j\in\{0,1,\ldots,n-1\}$,
\begin{align*}
 &i\int_0^t (a+it)^{n-j}G_j(e^{a+it}) d t
= \int_0^t(a+it)^{n-j}\cdot\frac{ d G_{j+1}(e^{a+it})}{ d t}d t\\
=&\left[ (a+it)^{n-j}G_{j+1}(e^{a+it})\right] _0^t-(n-j)i\int_0^t 
(a+it)^{n-j-1}G_{j+1}(e^{a+it})d t\\ 
=& (a+it)^{n-j}G_{j+1}(e^{a+it})-a^{n-j}G_{j+1}(e^a)
-(n-j)i\int_0^t (a+it)^{n-(j+1)}G_{j+1}(e^{a+it}) d t.
\end{align*}
Similarly,
\begin{gather*}
 i\int_0^t G_n(e^{a+it}) d t=G_{n+1}(e^{a+it})-G_{n+1}(e^a).
\end{gather*}
Hence there are constants $C_j\in\bC$ $(j=2,\ldots,n+1)$ and $C\in\bC$
such that for every $t\in\bR$,
\begin{align*}
 F(it)&
=\int_a^{a+it}z^nG_0(e^z) d z
=\int_0^t(a+it)^nG_0(e^{a+it})i d t\\
&=G_1(e^{a+it})(a+it)^n
+\sum_{j=1}^nC_{j+1}G_{j+1}(e^{a+it})(a+it)^{n-j}+C.
\end{align*}
Now the assumption $\min_{t\in\bR}|G_1(e^{a+it})|>0$ 
together with $\max_{t\in\bR}|G_j(e^{a+it})|<\infty$ $(j=2,\ldots,n+1)$
completes the proof.
\end{proof}

Next we consider asymptotic curves of $F$.
A curve ${\Gamma}:[0,1)\to\bC$ is called an asymptotic curve of an entire function $g$
with a finite asymptotic value $b\in\bC$ if 
${\Gamma}$ tends to $\infty$ and $\lim_{t\to 1}g\circ{\Gamma}(t)=b$.
We recall the following well-known

\begin{iversenth}[cf.\ \cite{N}]
 Let $g$ be a non-constant entire function. 
 Suppose that $z_0\in\bC$ is not a critical point of $g$, and put $w_0:=g(z_0)$. 
 Let $\phi$ be a single-valued analytic branch of $g^{-1}$ at $w_0$
 such that $\phi(w_0)=z_0$, and $\gamma:[0,1]\to \mathbb C$ be a curve with $\gamma (0)=w_0$.
 If the analytic continuation of $\phi$ along $\gamma|[0,t]$ is possible for any $t\in [0,1)$, 
 but impossible for $t=1$, then either 
\begin{itemize}
 \item $\lim_{t\to 1}\phi\circ\gamma(t)\in\bC$ exists and is a critical point of $g$, or 
 \item $\phi\circ\gamma$ 
 tends to $\infty$. In this case, $\phi\circ\gamma$
 is an asymptotic curve of $g$ with the finite asymptotic value $\gamma(1)$.
\end{itemize}
\end{iversenth}

For completeness, we include a proof.

\begin{proof}
We claim that the cluster set
$C:=\bigcap_{t\in[0,1)}\overline{\phi\circ\gamma([t,1))}$,
where the closure is taken in $\hat{\bC}$, is non-empty and
connected: indeed, from the compactness of $\hat{\bC}$, $C\neq\emptyset$.
If $C$ is not connected,
then there are distinct open subsets $U_1$ and $U_2$ in $\hat{\bC}$ 
intersecting $C$ such that $(C\cap U_1)\cup(C\cap U_2)=C$. 
There are $(t_j^1)$ and $(t_j^2)$ in $[0,1)$ 
tending to $1$ such that $\lim_{j\to 1}\phi(\gamma(t_j^i))$ exists in $C\cap U_i$
for each $i\in\{1,2\}$ and that for every $j\in\bN$, $t_j^1<t_j^2<t_{j+1}^1$.
For every $j\in\bN$, since $\phi\circ\gamma$ is continuous,
$\phi\circ\gamma([t_j^1,t_j^2])$ is connected. Hence there is $t_j\in[t_j^1,t_j^2]$ such that 
$\phi(\gamma(t_j))\in\hat{\bC}\setminus(U_1\cup U_2)$.
From the compactness of $\hat{\bC}$, there is a subsequence $(s_j)$
of $(t_j)$ tending to $1$ such that $\lim_{j\to 1}\phi(\gamma(s_j))$ exists 
in $C\setminus(U_1\cup U_2)$. This is a contradiction.
Thus $C$ is connected.

Unless $C$ is a singleton, $C$ is a continuum.
From 
$g(\phi\circ\gamma (t))=\gamma (t)$
for every $t\in [0,1)$ and the continuity of $g$,
$g(C)=\{\gamma(1)\}$. 
Then by the identity theorem, 
$g$ must be constant. This is a contradiction. 
Hence $C$ is a singleton, so $z_1:=\lim_{t\to 1}\phi(\gamma(t))\in\hat{\bC}$ exists. 
If $z_1\in\bC$, then $z_1$ is a critical point of $g$ 
since $\phi$ cannot be continued analytically along all over the $\gamma$.
If $z_1=\infty$, then $\phi\circ\gamma$ tends to $\infty$.
\end{proof}

For each $w_0\in \mathbb C$ and each $r>0$, put $\bD_r(w_0):=\{w\in\bC;|w-w_0|<r\}$.
Put $\bH_+:=\{z\in\bC;\Re z>0\}$, $I:=\{z\in\bC;\Re z=0\}$
and $\bH_-:=\{z\in\bC;\Re z<0\}$.

\begin{lemma}\label{lem:cont}
Let $\Gamma :[0,1) \to \mathbb C$ be an asymptotic curve of $F$ 
with a finite asymptotic value.
Then for every $t\in[0,1)$ close enough to $1$, we have $\Gamma(t)\in\bH_-$.
\end{lemma}

\begin{proof}
We begin with

\begin{claim}\label{th:infinite}
For any real-analytic curve $C:[0,1) \to \bH_+$ tending to $\infty$, 
the length of $F\circ C$ is infinite.
\end{claim}

\begin{proof}
 If the real part $\Re C$ of $C$ tends to $\infty$, then
 the curve $\pi\circ C$ also tends to $\infty$.
 Thus by \eqref{eqn:length}, 
 assumption \eqref{eq:complete} implies that
 the length of $F\circ C$ is infinite.
 If $M:=\sup\Re C<\infty$, then 
 $\pi\circ C\subset\{e^a<|\zeta|<e^{M+a}\}$ 
 and $\lim_{t\to 1}|\arg (\pi\circ C(t))|=\infty$. Hence
 \begin{gather}\label{eqn:length2}
 \int_{\pi\circ C}|f(\zeta)|| d\zeta|=\infty .
 \end{gather}
 Since $|\log \zeta|\geq a$ on the curve $\pi\circ C$, 
 (\ref{eqn:length2}) with 
 equality (\ref{eqn:length}) 
 implies that the length of $F\circ C$ is infinite.

 In the remaining case, 
 $C$ should transverse some vertical strip $\{b_1\le\Re z\le b_2\}$
 infinitely often.
 Then $\pi \circ C$ transverses a round annulus $\{e^{b_1+a}\le|\zeta|\le e^{b_2+a}\}$ 
 infinitely many times.
 Hence again we get \eqref{eqn:length2}, and the same argument as the above 
 implies that the length of $F\circ C$ is infinite.
\end{proof}

Let $w_1\in \mathbb C$ be the finite asymptotic value of $F$
along $\Gamma :[0,1) \to \mathbb C$, that is, $\lim_{t\to 1}F\circ\Gamma(t)=w_1$.

\begin{claim}\label{th:bounded}
 There exists $r>0$ such that any component of $F^{-1}(\bD_{r}(w_1))$ 
 which intersects $I$ is bounded.
\end{claim}
\begin{proof}
 By Lemma \ref{th:imaginary}, there is $R>0$ such that 
 $\min_{s\in\bR,|s|\ge R}|F(is)|\ge|w_1|+1$.
 Increasing $R>0$ if necessary, we assume that $w_1\not\in F(\{|z|=R\})$, 
 so there is $r\in(0,1)$ such that $\bD_{r}(w_1)\cap F(\{|z|=R\})=\emptyset$.
 Then $F^{-1}(\bD_{r}(w_1))$ intersects neither $I\cap \{ |z|\geq R\}$ nor $\{|z|=R\}$,
 so
 any component of $F^{-1}(\bD_{r}(w_1))$ intersecting with $I$ is 
 contained in $\{|z|<R\}$.
\end{proof}
Fix $r>0$ with the property claimed above.
Fix $t_0 \in[0,1)$ such that
\begin{gather}
 F\circ \Gamma([t_0,1))\subset\bD_{r/4}(w_1).\label{eq:small}
\end{gather}
Let $\Omega$ be a component of $F^{-1}(\bD_{r}(w_1) )$ which contains $\Gamma (t_0)$.
Then $\Omega$ contains the whole $\Gamma ([t_0,1))$, so
$\Omega$ is unbounded. Hence by Claim \ref{th:bounded}, 
$\Omega$ does not intersect with $I$, so is contained in
either $\bH_-$ or $\bH_+$.

Assume contrary that the conclusion of the lemma does not hold.
Then $\Omega \subset \bH _+$.
Let $\phi$ be a germ of a single-valued analytic branch of $F^{-1}$ at 
$F( \Gamma (t_0))$ such that $\phi (F(\Gamma (t_0)))=\Gamma (t_0)$.
Then $\phi$ is holomorphic on the disc $\bD_{r/2}(F(\Gamma (t_0)))$, 
or else there exists a largest disk $\bD_{\rho}(F(\Gamma (t_0)))$ 
with $\rho\in(0,r/2)$ to which $\phi$ can be extended analytically. 

But the latter cannot occur: 
for there 
would then be a point $\xi\in \partial \bD_{\rho}(F(\Gamma (t_0)))$
over which $\phi$ cannot extend analytically.
Let $\alpha$ be the radial segment $[0,1]\ni s\mapsto 
F(\Gamma (t_0))+s(\xi-F(\Gamma (t_0)))\in\bD_{r/2}(F(\Gamma (t_0)))$ 
joining $F(\Gamma (t_0))$ and $\xi$. 
Since $\bD_{r/2}(F(\Gamma (t_0)))\subset\bD_{3r/4}(w_1)$, 
the curve $\phi\circ\alpha|[0,1)$ is contained in $\Omega$, so in $\bH_+$.
Since the unique critical point $-a$ of $F$ is in $\bH_-$,
Iversen's theorem 
yields that
the curve $\phi\circ\alpha|[0,1)$ tends to $\infty$.
On the other hand, $\phi\circ\alpha|[0,1)$ is real-analytic 
and the length of $F\circ(\phi\circ\alpha|[0,1))=\alpha|[0,1)$ is finite,
so by Claim \ref{th:infinite}, $\phi\circ\alpha|[0,1)$ cannot tend to $\infty$.
This is a contradiction.

Thus $\phi$ is holomorphic on $\bD_{r/2}(F(\Gamma (t_0)))$,
which contains $F\circ\Gamma([t_0,1))$ by \eqref{eq:small}.
Hence $\lim_{t\to 1}\Gamma (t)=\phi(\lim_{t\to 1}F\circ\Gamma(t))
=\phi (w_1)$, which contradicts that
$\Gamma$ tends to $\infty$.

Now the proof is complete.
\end{proof}

For a domain $D$ in $\bC$, a subset $c$ in $D$ is called
a crosscut (or a transverse arc) of $D$ if $c$ is homeomorphic to $(0,1)$,
the closure $\overline{c}$ in $\bC$ is homeomorphic to $[0,1]$ and 
$\overline{c}\cap\partial D$ consists of two points.

For each $r>0$, put $\bD_r:=\bD_r(0)=\{w\in\bC;|w|<r\}$. 

\begin{lemma}\label{lem:2}
For every $R>0$, $F^{-1}(\bD_R)\cap\bH_+$ has no unbounded components.
\end{lemma}

\begin{proof}
Let $\Omega$ be a component of $F^{-1}(\bD_R)\cap\bH_+$. From Lemma \ref{th:imaginary}, 
 $(\partial\Omega)\cap I$ has at most finitely many components, 
 which are closed intervals. 
The image of each component of $(\partial\Omega)\cap I$ under $F$
 is a 
real-analytic curve in 
$\overline{\bD_R}$, and
 $\bD_R\setminus F((\partial\Omega)\cap I)$ has at most finitely many components.
Fix a triangulation of $\overline{\bD_R}$ 
 such that the interior of any triangle is contained in $\bD_R\setminus F((\partial\Omega)\cap I)$.

As convention, we call the interior of each triangle an {\itshape open triangle}.

\setcounter{claim}{0}
\begin{claim}
 For every open triangle $V$ and 
 every component $U$ of $F^{-1}(V)\cap\Omega$, 
 $U$ is bounded and
 the restriction $F_{\overline{U}}$ of $F$ on $\overline{U}$ is a homeomorphism 
 from $\overline{U}$ onto $\overline{V}$.
\end{claim}

\begin{proof}
 Fix $z_0\in U$. By $F'(z_0)\neq 0$, there is a germ $\phi$ of a single-valued branch of 
 $F^{-1}$ with $\phi(F(z_0))=z_0$.
 Assume that there is a curve $\gamma :[0,1]\to \overline{V}$ with $\gamma (0)=F(z_0)$ 
 such that the analytic continuation of $\phi$ along $\gamma |[0,t]$
 is possible for any $t\in [0,1)$, but impossible for $t=1$.

Since the unique critical point $-a$ of $F$ is in $\bH_-$,
by Iversen's theorem, the curve $\phi\circ\gamma$ is an asymptotic curve of $F$
with the finite asymptotic value $\gamma(1)\in \mathbb C$. 
Then by Lemma \ref{lem:cont},
 there is $t_0\in[0,1)$ such that
 $\phi\circ\gamma(t_0)\in\bH_-$. 
On the other hand, from $F(I)\cap V=\emptyset$, 
$U$ is a component of $F^{-1}(V)$.
Moreover, $\overline{U}$ is a component of $F^{-1}(\overline{V})$
since there is no critical point of $F$ on $I$.
Thus the curve $\phi\circ\gamma$ is in $\overline{U}$, so in $\bH_+\cup I$.
This contradicts that $\phi\circ\gamma(t_0)\in\bH_-$.
 
We have shown that $\phi$ extends analytically along all curves in $\overline{V}$.
 Now by the monodromy theorem, 
 a single-valued continuous branch $F^{-1}:\overline{V}\to \overline{U}$ exists. 
 Hence $U$ is bounded and $F_{\overline{U}}:\overline{U}\to\overline{V}$ is homeomorphic.
\end{proof}

Let $N$ be the number of triangles in $\overline{\bD_R}$.

\begin{claim}\label{th:increasing}
 There is an increasing sequence of closed sets
\begin{gather*}
 D_1\subset D_2\subset \cdots \subset D_N=\overline{\bD_R}
\end{gather*}
 such that for each $j\in\{1,\ldots,N\}$,
 $D_j$ consists of $j$ triangles and
 $\Int D_j$ is connected and simply connected. 
\end{claim}

\begin{proof}
 This is clear if $N=1$, so we assume that $N\ge 2$.
 The construction is decreasingly inductive.
 For $j=N$, $D_N=\overline{\bD_R}$ 
 consists of $N$ triangles and $\Int D_N=\bD_R$ is connected and simply connected. 
 Fix $j\in\{1,\ldots,N-1\}$,
 and suppose that we obtain a closed set $D_{j+1}$ consisting of $j+1$ triangles 
 such that $\Int D_{j+1}$ is connected and simply connected.

 Let $\mathcal{S}_j$ 
 be the set of all triangles $\Delta$ in $D_{j+1}$
 having an edge in $\partial D_{j+1}$
 such that $\Int(D_{j+1}\setminus\Delta)$ is not connected.
 Let us find a triangle $\Delta_j$ in $D_{j+1}$ 
 which has an edge in $\partial D_{j+1}$ and does not belong to $\mathcal{S}_j$. 
 We can certainly do this when $\mathcal{S}_j=\emptyset$.
 Suppose that $\mathcal{S}_j\neq\emptyset$. For each $\Delta\in\mathcal{S}_j$,
 there are two components $P$ and $P'$ of 
 $\Int( D_{j+1}\setminus\Delta)$
 and put $N(\Delta)$ be the minimum of the number of triangles in $\overline{P}$ 
 and that of $\overline{P'}$. Fix a triangle $\Delta\in\mathcal{S}_j$
 satisfying
\begin{gather}
 N(\Delta)=\min_{\Delta'\in\mathcal{S}_j}N(\Delta'),\label{eq:minimal}
\end{gather} 
and a component $P$ of $\Int(D_{j+1}\setminus\Delta)$
 such that $\overline P$ consists of $N(\Delta)$ triangles. 
 Then any triangle $\Delta_j$ in $\overline{P}$ having an edge in 
 $(\partial D_{j+1})\cap \overline{P}$
 will not belong to $\mathcal{S}_j$: for, if $\Delta_j\in\mathcal{S}_j$, then 
 there is a component of 
 $\Int(D_{j+1}\setminus\Delta_j)$,
 which is a subset of $\Int(P\setminus\Delta_j)$,
 so $N(\Delta_j)<N(\Delta)$. This contradicts \eqref{eq:minimal}.

 With such $\Delta_j$, set $D_j:=\overline{D_{j+1}\setminus\Delta_j}$. Then $\Int D_j$ is 
 connected, and moreover $(\partial{\Delta_j})\cap(\Int D_{j+1})$
 is a crosscut of $\Int D_{j+1}$, so $\Int D_j$ is simply connected.
\end{proof}

Let $(D_j)$ be the increasing sequence of closed sets 
obtained in Claim $\ref{th:increasing}$.
We show by induction that
for each $j\in\{1,\ldots,N\}$,
$F^{-1}(\Int D_j)\cap\bH_+$ has no unbounded components.

 For $j=1$, $D_1$ 
 is a single triangle. This case is covered by Claim 1.

 Fix $j\in\{1,\ldots,N-1\}$, and suppose the assertion holds for $D_j$.
 Put $C:=(\partial D_j)\cap\Int D_{j+1}$, which is a crosscut of $\Int D_{j+1}$.
 Assume that a component $\Omega_0$ of $F^{-1}(\Int D_{j+1})\cap\bH_+$ is unbounded.
 Let $c_1$ be a component of $F^{-1}(C)\cap \Omega _0$.
 Since $F$ has no critical point in $\Omega _0$, $c_1$ is a 
 crosscut of $\Omega _0$, and
 $\Omega_0\setminus c_1$, which is possibly still connected,
 has an unbounded component $\Omega _1$.

 Let $U_1$ be a component of $\Omega _1\setminus F^{-1}(C)$ such that $c_1\subset\partial U_1$.
 Then $U_1$ is a component of 
 either $F^{-1}(\Int D_{j})\cap\Omega$ or $F^{-1}(\Int \Delta_{j})\cap\Omega$.
 In either case, by the assumption for $j$ and the assertion for $j=1$, 
 $U_1$ is bounded. Hence $\Omega_1\setminus\overline{U_1}$ has
 an unbounded component $\Omega_2$. Let $c_2$ be a component of 
 $\partial \Omega _2\cap\Omega_1$.
 Then $c_2$ is a crosscut of $\Omega_1$ and $c_2\subset\partial U_1$.

 Let $U_2$ be a component of $\Omega_2\setminus F^{-1}(C)$ such that $c_2\subset\partial U_2$.
 Then $U_2$ is bounded by the same reason.
 Let $c_3$ be a component of $\partial U_2\cap\Omega _2$,
 which is a crosscut of $\Omega_2$.

 Now $c_2\subset\partial U_1\cap\partial U_2$.
 Hence at least one of $U_1$ and $U_2$, say $U_*$, is a component of $F^{-1}(\Int\Delta_j)\cap\Omega$.
 By Claim 1, 
 $F$ restricts to a homeomorphism from $\partial(U_*)$ to $\partial\Delta_j$.
 But $F^{-1}(C)\cap \partial(U_*)$ contains not only $c_2$
 but also either $c_1$ or $c_3$, which is a contradiction.

 Hence $F^{-1}(\Int D_{j+1})\cap\bH_+$ also has no unbounded components.
 This completes the induction.

This applies in particular to $\Int D_N=\bD_R$.
Hence $F^{-1}(\bD_R)\cap\bH_+$ has no unbounded components, which completes the
proof.
\end{proof}

\begin{lemma}\label{th:left}
For every $z\in\bH_-\cup I$,
$|F(z)|\le\max\{|z|,|a|\}^{n+1}\cdot 2^{n}\max_{|\zeta|\le e^a}|\zeta f(\zeta ) |$.
\end{lemma}

\begin{proof}
For each $z\in\bH_-\cup I$,
\begin{gather*}
 |F(z)|=\left|\int_{[a,a+z]}z^nf(e^z)e^z d z\right|
\le (|z|+|a|)^n|z|\max_{|\zeta|\le e^a}|\zeta f(\zeta ) |,
\end{gather*}
where $[a,a+z]$ is the closed segment joining $a$ and $a+z$.
\end{proof}

\begin{lemma}\label{lem:4}
For each $r>0$, put 
\begin{gather*}
\mu_+(r):=\min\{|F(z)|;z\in\bH_+\cup I, |z|=r\}. 
\end{gather*}
If $f$ is transcendental, then
$\liminf_{r\to\infty}
\mu_+(r)\le 1$.
\end{lemma}

\begin{proof}
For every $r>0$, put
\begin{gather*}
 D_r:=\{r/4<|z|<2r\} \cap \bH_+.
\end{gather*}
Then $D_1\cong D_r$ under the similarity $z\mapsto rz$. 
Let $\varphi:\bD_1\to D_1$ be a (inverse of) Riemann mapping such that
$\varphi(0)=1\in D_1$. 
For every $r>0$, the conformal map 
\begin{gather*}
 \varphi_r:=r\cdot\varphi:\bD_1\to D_r 
\end{gather*}
satisfies that $\varphi_r(0)=r\in D_r$
and extends to a homeomorphism from $\overline{\bD_1}$ onto $\overline{D_r}$.
The Poisson kernel on $\bD_1$ is
\begin{gather*}
 P(w,\xi):=\Re\left(\frac{\xi+w}{\xi-w}\right)=\frac{1-|w|^2}{|\xi-w|^2}
\end{gather*} 
for $w\in\bD_1$ and $\xi\in\partial{\bD_1}$. 
For each $w\in\bD_1$, $P(w,\xi)|d\xi|/(2\pi)$ is a probability measure on $\partial\bD_1$,
and more specifically, the harmonic measure for $\bD_1$ with pole at $w$
(for the details, see, e.g., \cite[\S1.2]{R}).

Assume that 
\begin{gather*}
 \liminf_{r\to\infty}\mu_+(r)>1. 
\end{gather*}
Then there is $r_0>0$ such that
$\log|F|$ is positive and harmonic on $\{|z|>r_0\}\cap(\bH_+\cup I)$. 

Let us compare $\log|F(r)|$ and $\log|F(r/2)|$ for each $r>4r_0$.
Since $\log|F\circ \varphi _r|$ is positive and harmonic on $\bD_1$,
Harnack's inequality (cf. \cite[Theorem 1.3.1]{R})
yields
\begin{gather*}
 \frac{1-|\varphi^{-1}(1/2)|}{1+|\varphi ^{-1}(1/2)|}\log |F(r)|\leq \log |F(r/2)|
\end{gather*}
(we note that $\log|F(r)|=\log|F\circ \varphi _r(0)|$ and
that $\log |F(r/2)|=\log|F\circ \varphi _r(\varphi^{-1}(1/2))|$). 
Hence we have
\begin{gather*}
 \log |F(r)|\le C_0\log |F(r/2)|,
\end{gather*}
where we put $C_0:=\frac{1+|\varphi ^{-1}(1/2)|}{1-|\varphi^{-1}(1/2)|}>1$.

A repeated use of this estimate implies that
 \begin{gather}
\log |F(r)|\le C_0^{\max\{j\in\bN;r/2^j>2r_0\}}\cdot\max_{s\in[2r_0,4r_0]}\log|F(s)|
\le C_1r^{\alpha},\label{eq:real} 
 \end{gather} 
where we put $\alpha:=\log_2 C_0>0$ and 
$C_1:=(2r_0)^{-\alpha}\max_{s\in[2r_0,4r_0]}\log|F(s)|>0$.

Let us next compare $\log|F(z)|$ and $\log|F(|z|)|$ for each $z\in\bH_+$ with $|z|>4r_0$.
Fix $z\in\bH_+$ with $|z|>4r_0$ and put $r=|z|$. Then $z\in D_r$.
Let us decompose $\partial D_r$ into the disjoint subsets
$I_r:=(\partial D_r)\cap I$ and 
$S_r:=(\partial D_r)\setminus I$. 
Then $\varphi_r^{-1}(I_r)=\varphi ^{-1}(I_1)$ and $\varphi_r^{-1}(S_r)=\varphi^{-1}(S_1)$, and
\begin{multline}
 \log |F(z)|=\int_{\varphi ^{-1}(I_1)}\left( \log|F(\varphi_r(\xi))|\right) 
P(\varphi_r^{-1}(z),\xi)\frac{|d\xi|}{2\pi} \\
 +\int_{\varphi ^{-1}(S_1)}\left(\log|F(\varphi_r(\xi))|\right) P(\varphi_r^{-1}(z),\xi)\frac{|d\xi|}{2\pi}.
\label{eq:harmonicmeasure}
\end{multline}
Increasing $r_0>0$ if necessary, Lemma \ref{th:left} implies 
that
$$\log|F(it)|\le 2\log(|t|^{n+1})=2(n+1)\log|t|$$
for every 
$t\in \bR$ with $|t|>r_0$. Since $I_r\subset\{it\in\bR;|t|<2r\}$,
\begin{gather}
 \int_{\varphi ^{-1}(I_1)}\left(\log|F(\varphi_r(\xi))|\right) P(\varphi_r^{-1}(z),\xi)\frac{|d\xi|}{2\pi}\le 2(n+1)\log(2r).\label{eq:imaginary}
\end{gather}
Put $c:=\varphi_r^{-1}(\{z\in D_r;|z|=r\})$,
which is a crosscut of $\bD_1$. Note that
\begin{gather*}
 \{\eta\xi^{-1}\in\overline{\bD_1};\eta\in\overline{c},
 \xi\in\varphi^{-1}(\overline{S_1})\} 
\end{gather*}
is compact in $\overline{\bD_1}$ and does not contain $1$. Put
\begin{gather*}
 C_2:=\max\{P(\eta\xi^{-1},1);\eta\in\overline{c},
\xi\in\varphi^{-1}(\overline{S_1})\}< \infty.
\end{gather*}
We note that $P(\eta,\xi)=P(\eta\xi^{-1},1)$ for every $\xi\in\partial\bD_1$.
Since $\varphi_r ^{-1}(z)\in c$ and $\log|F|\ge 0$ on $\partial D_r$, 
we have
\begin{gather*}
\int_{\varphi ^{-1}(S_1)}\left( \log |F(\varphi_r(\xi))| \right) 
P(\varphi_r^{-1}(z),\xi)\frac{|d\xi|}{2\pi}\le
C_2\int_{\partial \bD_1}\log|F(\varphi_r(\xi))|\frac{|d\xi|}{2\pi}
=C_2\log |F(|z|)|,
\end{gather*}
where the final equality follows from the mean value property of harmonic functions 
(we note that $\log|F(\varphi_r(0))|=\log|F(r)|$ and 
that $r=|z|$).
This with \eqref{eq:harmonicmeasure} and \eqref{eq:imaginary} concludes
\begin{gather*}
 \log |F(z)|\le 2(n+1)\log|2z|+C_2\log|F(|z|)|.
\end{gather*}

From this estimate with \eqref{eq:real}, on $\bH_+$, 
\begin{gather*}
 \log^+|F(z)|=O(\log|z|)+O(|z|^{\alpha})
\end{gather*}
as $|z|\to \infty$. This with Lemma \ref{th:left} implies that the order of $F$ is finite.
Thus by the definition of $F$, the order of $f(e^{z+a})$ is also finite.

On the other hand, 
we can show that the order of $f(e^{z+a})$ is infinite, which will prove our lemma by contradiction.
Since $f(\bC\setminus\{0\})\subset\bC\setminus\{0\}$, 
we can write as $f(\zeta)=\zeta^ke^{H(\zeta)}$ with
some $k\in\bN\cup\{0\}$ and some entire function $H(\zeta)$.
By the assumption that $f$ is transcendental, $H$ is non-constant.
Hence by Hadamard's theorem (cf. \cite[p.\ 209]{A}), the order of $f$ 
is greater than or equal to one.
Hence the order of $f(e^{z+a})$ is infinite.
This is a contradiction.

Thus we have proved  $\liminf_{r\to\infty}\mu_+(r)>1$.
\end{proof}

Let us complete the proof of Theorem.

Assume that $f$ is transcendental.
Fix $R_1>\max\{1,|F(-a)|\}$. Then by Lemma \ref{lem:4}, $F^{-1}(\bD_{R_1})\cap \bH_+$ is unbounded,
and then by Lemma \ref{lem:2}, there are infinitely many (bounded) components of 
$F^{-1}(\bD_{R_1})\cap\bH_+$. By Lemma \ref{th:imaginary}, 
the boundaries of at most finitely many components of $F^{-1}(\bD_{R_1})\cap \bH_+$
intersect $I$, so the other (infinitely many) components of $F^{-1}(\bD_{R_1})\cap\bH_+$ 
are all relatively compact in $\bH_+$.
Let $V$ and $W$ be distinct such components. 
Then $\overline{V}\cap\overline{W}=\emptyset$ since $F$ has no critical point in $\bH_+$.
Join $\overline{V}$ and $\overline{W}$
by a compact line segment $l$, take $R'_1>\max_{z\in l}|F(z)|(\ge R_1)$
and let $\Omega_1$ be the component of $F^{-1}(\bD_{R'_1})$ such that $l\subset\Omega_1$.
Then $\overline{V}\cup\overline{W}\subset\Omega_1\cap F^{-1}(\overline{\bD_{R_1}})$. Put 
$A_1:=\{R_1<|w|<R'_1\}$ and
\begin{gather*}
 \Omega_1':=\Omega_1\setminus F^{-1}(\overline{\bD_{R_1}}).
\end{gather*}
Then $\Omega_1'$ is a component of $F^{-1}(A_1)$ and is at least triply-connected. 
The restriction 
\begin{gather*}
 F_{\Omega_1'}:\Omega_1'\to A_1 
\end{gather*}
is locally homeomorphic, i.e,
has no critical point since $F(-a)\not\in A_1$. If 
$F$ has also no asymptotic curve in $\Omega_1'$ with a finite asymptotic value
in $A_1$,
then Iversen's theorem concludes that $F_{\Omega_1'}$ have the curve lifting property, 
that is, any closed curve may be lifted uniquely under $F_{\Omega'}$ 
given any preimage of the initial point
(for the details, see, e.g.,\ \cite[Definition 4.13]{Forster}),
and then
by \cite[Theorem 4.19]{Forster}, the local homeomorphism $F_{\Omega_1'}$ must be a covering.
Since the universal covering of $A_1$ is topologically a disk and 
$\pi_1(A_1)$ is $\bZ$, $\pi_1(\Omega_1')$ must be either $\bZ$ or $\{1\}$, so
$\Omega_1'$ must be topologically either an annulus or a disk. 
This contradicts that $\Omega_1'$ is at least triply connected.

Hence $F$ has an asymptotic curve ${\Gamma_1}\subset\Omega_1'$
with a finite asymptotic value $a_1$ in $A_1=\{R_1<|w|<R'_1\}$.
By Lemma \ref{lem:cont}, we may assume that ${\Gamma_1}\subset\Omega_1'\cap\bH_-$.

Fix $R_2>R'_1$. By the same argument applied to $R_2$, we obtain $R_2'>R_2$,
a component $\Omega_2'$ of $F^{-1}(\{R_2<|w|<R_2'\})$ and
an asymptotic curve ${\Gamma_2}\subset \Omega _2'\cap\bH_-$ of $F$
with a finite asymptotic value $a_2\in \{ R_2<|w|<R_2'\}$.
We note that ${\Gamma_1\cap\Gamma_2}=\emptyset$. Without loss of generality,
we assume that both {$\Gamma_1$ and $\Gamma_2$} are simple.
Let $U$ be an unbounded domain contained in $\bH_-$ such that 
\begin{gather*}
 \partial U={\Gamma_1\cup\Gamma_2}\cup c, 
\end{gather*}
where $c$ is an arc joining the endpoint of {$\Gamma_1$} and that of {$\Gamma_2$} in $\bH_-$.
Since $a_1\neq a_2$, by Lindel\"of's theorem (cf.\ \cite[p.\ 65]{N}), 
\begin{gather*}
 \sup_U|F|=\infty.
\end{gather*}

On the other hand, 
we can show that $|F|$ is bounded on U, which will prove the Theorem by contradiction.
First, we note that
\begin{gather*}
 M:=\max_{z\in\partial U}\log |F(z)|<\infty.
\end{gather*}

For a bounded domain $D\subset \bC$,
let $(z,E)\mapsto\omega_D(z,E)$ be the harmonic measure of $D$,
where $z\in D$ and $E\subset \partial D$ is a Borel subset.
For the details, see, e.g., \cite[\S4.3]{R}.

Fix $z_0\in U$. 
For every $r>|z_0|$, let $U_r$ be the component of $U\cap \bD_r$ which contains $z_0$.
Then by the two constant theorem (cf.\ \cite[p.\ 101]{R}), 
\begin{gather*}
\log |F(z_0)|\le 
M\omega_{U_r}(z_0,\partial U_r\setminus\partial\bD_r )
+\left(\sup_{z\in \bD_r\cap\bH_-}\log |F(z)|\right)\omega_{U_r}(z_0,\partial U_r\cap \partial \bD_r)\\
\leq M\cdot 1+\left( \sup_{z\in \bD_r\cap\bH_-}\log |F(z)|\right)\omega_{U_r}(z_0,\partial U_r\cap \partial \bD_r),
\end{gather*}
so we have
\begin{gather}
\log |F(z_0)|\le M+\limsup_{r\to \infty}
 \left(  \left( \sup_{z\in \bD_r\cap\bH_-}\log |F(z)| \right) \omega_{U_r}(z_0,\partial U_r\cap \partial \bD_r)\right). \label{eq:Lindelof}
\end{gather}

By the monotonicity of harmonic measures (cf.\ \cite[Corollary 4.3.9]{R}),
\begin{gather*}
 \omega_{U_r}(z_0,\partial U_r\cap \partial \bD_r)
 \le\omega_{\bD_r\cap\bH_-}(z_0,\partial U_r\cap\partial\bD_r) 
 \le\omega_{\bD_r\cap\bH_-}(z_0,\partial \bD_r\cap\bH_-)
 =\frac{2}{\pi}\arg\frac{ir+z_0}{ir-z_0}
\end{gather*}
(for the final equality, cf.\ \cite[p.\ 100]{R}).
Hence as $r\to \infty$, 
\begin{gather*}
 \omega_{U_r}(z_0,\partial U_r\cap \partial \bD_r)=O(r^{-1}).
\end{gather*}
This with Lemma \ref{th:left} implies that
\begin{gather*}
 \limsup_{r\to \infty}
\left( \left( \sup_{z\in \bD_r\cap\bH_-}\log |F(z)| \right)\omega_{U_r}(z_0,\partial U_r\cap \partial \bD_r)\right) \leq 0.
\end{gather*}
Hence by \eqref{eq:Lindelof}, we conclude $\sup_U\log |F|\le M$ since $z_0\in U$ is arbitrary.
Now the proof is complete.
\end{proof}

\begin{acknowledgement}
 Katsutoshi Yamanoi thanks Professors Masaaki Umehara and Kotaro Yamada for turning his attention to \cite[Question C]{UY}. The authors thank Professors Masaaki Umehara, Kotaro Yamada
 and David Drasin for invaluable comments. 
 Y\^usuke Okuyama was partially supported by JSPS Grant-in-Aid for Young Scientists (B), 21740096.
Katsutoshi Yamanoi
 was partially supported by JSPS Grant-in-Aid for Young Scientists (B), 20740077
 and by the foundation of Challenging Research Award, Tokyo Institute of Technology.
\end{acknowledgement}


\begin{thebibliography}{1}

\bibitem{A}
{\sc Ahlfors,~L.~V.} {\em Complex analysis}, McGraw-Hill Book Co., New York,
  third edition (1978), An introduction to the theory of analytic functions of
  one complex variable, International Series in Pure and Applied Mathematics.

\bibitem{Forster}
{\sc Forster,~O.} {\em Lectures on {R}iemann surfaces}, Vol.~81 of {\em
  Graduate Texts in Mathematics}, Springer-Verlag, New York (1991), Translated
  from the 1977 German original by Bruce Gilligan, Reprint of the 1981 English
  translation.

\bibitem{Huber57}
{\sc Huber,~A.} On subharmonic functions and differential geometry in the
  large, {\em Comment. Math. Helv.}, {\bf 32} (1957), 13--72.

\bibitem{N}
{\sc Nevanlinna,~R.} {\em Analytic functions}, Translated from the second
  German edition by Phillip Emig. Die Grundlehren der mathematischen
  Wissenschaften, Band 162, Springer-Verlag, New York (1970).

\bibitem{O}
{\sc Osserman,~R.} {\em A survey of minimal surfaces}, Dover Publications Inc.,
  New York, second edition (1986).

\bibitem{R}
{\sc Ransford,~T.} {\em Potential theory in the complex plane}, Vol.~28 of {\em
  London Mathematical Society Student Texts}, Cambridge University Press,
  Cambridge (1995).

\bibitem{UY}
{\sc Umehara,~M.{\rm\ and }Yamada,~K.} Applications of a completeness lemma in
  minimal surface theory to various classes of surfaces, {\em Bull. Lond. Math.
  Soc.}, {\bf 43}, 1 (2011), 191--199.


\bibitem{UYc}
{\sc Umehara,~M.{\rm\ and }Yamada,~K.} Corrigendum: Applications of a completeness lemma in
  minimal surface theory to various classes of surfaces, {\em Bull. Lond. Math.
  Soc.}, {\bf 44}, 3 (2012), 617--618.


\end{thebibliography}

\end{document}